\documentclass[a4paper]{amsart} 
\usepackage{enumerate}
\usepackage{amsmath}
 \ifx\pdftexversion\undefined
 \usepackage[dvips]{graphicx,xcolor}
 \else
 \usepackage[pdftex]{graphicx,xcolor}
 \fi

\let\ifanglais\iftrue

\def\R{{\mathbb R}}
\def\N{{\mathbb N}}

\definecolor{teal}{RGB}{0,128,128}

\newcommand{\htarea}{\operatorname{Area}}
\newcommand{\barea}{\operatorname{Area}^B}
\newcommand{\vol}{\operatorname{Vol}}
\newcommand{\htvol}{\operatorname{Vol^{H}}}
\newcommand{\bvol}{\operatorname{Vol^B}}

\newcommand\cone{L}
\newcommand\intersection{\cap}
\newcommand\bary{T}

\newcommand{\ed}[1]{\textrm{d} #1}

\newcommand{\asympball}{\operatorname{AsB}}
\newcommand{\ball}{B}
\newcommand{\euclidean}{\mathcal{E}}

\newcommand{\interior}{\operatorname{int}}

\newcommand{\assvol}{\operatorname{Asvol}}
\newcommand{\flags}{\operatorname{Flags}}
\newcommand{\thedim}{d}
\newcommand{\dsimplex}{\Sigma}

\newcommand\cardinal[1]{\lvert{#1}\rvert}

\newcommand\flagapprox{a_{\textnormal{f}}}
\newcommand\flagnumber{N_{\textnormal{f}}}

\newcommand\shrunkP{P}
\renewcommand\epsilon{\varepsilon}

\newtheoremstyle{mesthm}
  {10pt plus 1pt minus 1pt}
  {9pt plus 1pt minus 6pt}
  {\slshape}
  {0.5cm}
  {\bfseries}
  {.}
  {1ex}
  {}
\newtheoremstyle{mesdefi}
  {6pt plus 1pt minus 1pt}
  {6pt plus 1pt minus 1pt}
  {}
  {0.5cm}
  {\bfseries}
  {.}
  {1ex}
  {}%

\theoremstyle{mesthm}
\newtheorem{lema}{\ifanglais{\large L}emma\else{\large L}emme\fi}
\newtheorem{theo}[lema]{\ifanglais{\large T}heorem\else {\large
    T}h\'eor\`eme\fi}
\newtheorem*{theo*}{\ifanglais{\large T}heorem\else {\large T}h\'eor\`eme\fi}

\newtheorem{cor}[lema]{{\large C}orollary}
\newtheorem*{cor*}{{\large C}orollary}

\theoremstyle{mesdefi}
\newtheorem{defi}[lema]{\ifanglais{\large D}efinition\else{\large
    D}\'efinition\fi} 
\newtheorem*{defi*}{\ifanglais Definition\else D\'efinition\fi}

\newcounter{step}

\DeclareMathOperator{\ent}{Ent}

\title[Flag-approximability and volume entropy]%
{Flag-approximability of convex bodies \\
and volume growth of Hilbert geometries}
\author[C.~Vernicos and C.~Walsh]{Constantin Vernicos and Cormac Walsh}
\address{%
  Institut Montpellierain Alexander Grothendieck\\ 
  Universit\'e de Montpellier \\
  Case Courrier 051\\
  Place Eug\`ene Bataillon \\
  F--34395 Montpellier Cedex\\ 
  France} 
\email{Constantin.Vernicos@math.univ-montp2.fr}
\address{%
  Inria \\
  CMAP, Ecole polytechnique \\
  CNRS \\
  Universit\'e Paris-Saclay,
  91128 Palaiseau \\
  France}
\email{cormac.walsh@inria.fr}


\thanks{The authors acknowledge that this material is based upon work
partially supported by the ANR Blanche ``Finsler'' grant}

\begin{document}

\begin{abstract}
We introduce the flag-approximability of a convex body to measure
how easy it is to approximate by polytopes.
We show that the flag-approximability is exactly half
the volume entropy of the Hilbert geometry on the body,
and that both quantities are maximized when the convex body
is a Euclidean ball.

We also compute explicitly the asymptotic volume of a convex polytope,
which allows us to prove that simplices have the least asymptotic volume.
\end{abstract}

\maketitle

\section*{Introduction}

An important problem with many practical applications is to approximate
convex bodies with polytopes that are as simple as possible, in some sense.
Various measures of complexity of a polytope have been considered in the
literature. These include counting the number of vertices, the number of
facets, or even the number of faces~\cite{arya_da_fonseca_mount_journal}.
One could also use, however, the number of \emph{maximal flags}.
Recall that a maximal flag of a $d$-dimensional polytope is
a finite sequence $(f_0, \dots, f_d)$ of faces of the polytope
such that each face $f_i$ has dimension $i$ and is contained in the
boundary of $f_{i+1}$.

Suppose we wish to approximate a convex body $\Omega$ by a polytope
within a Hausdorff distance $\epsilon > 0$.
Let $\flagnumber(\epsilon, \Omega)$ be the least number of maximal flags
over all polytopes satisfying this criterion.
We define the
\textsl{flag approximability} of $\Omega$ to be
\begin{align*}
\flagapprox(\Omega)
    := \liminf_{\epsilon\to 0}
        \frac{\log \flagnumber(\epsilon, \Omega)}{-\log\epsilon}.
\end{align*}
This is analogous to how Schneider and Wieacker~\cite{sw}
defined the (vertex) approximability, where the least number of
vertices was used instead of the least number of maximal flags.

It is not known which if any equalities hold between the vertex, facet, face,
and flag approximabilities. An advantage of using the flag approximability
is that one can prove the following relation to
the \textsl{volume entropy} of the Hilbert metric on the body.

Choose a base point $p$ in the interior of the convex body $\Omega$,
and for each $R>0$ denote by $\ball_\Omega(p, R)$ the closed ball
centered at $p$ of radius $R$ in the Hilbert geometry.
Let $\htvol$ denote the Holmes--Thompson volume.
The (lower) volume entropy of the Hilbert geometry on $\Omega$ is defined to be
\begin{align*}
\ent(\Omega)
    := \liminf_{R\to\infty}
           \frac{\log \htvol\bigl(\ball_\Omega(p, R)\bigr)}{R}.
\end{align*}

Observe that this does not depend on the base point $p$, and moreover
does not change if one takes instead the Busemann volume.
One can also define the upper flag approximability and the upper
volume entropy by taking supremum limits instead of infimum ones.
Although the two entropies do not generally coincide, as shown by the first
author in \cite{ver2014}, all our results and proofs hold
when replacing $\liminf$ with $\limsup$.

\begin{theo}
\label{thm:entropy_flag_approximabilty}
Let $\Omega \subset \R^\thedim$ be a convex body. Then,
\begin{equation*}
\ent(\Omega) = 2\flagapprox(\Omega).
\end{equation*}
\end{theo}

The same result concerning the vertex approximability was proved
by the first author~\cite{ver2014} in dimensions two and three.
In higher dimension, it was shown only that the
volume entropy is greater than or equal to twice the vertex approximability.
The motivation was to try to prove the entropy upper bound conjecture,
which states that the volume entropy of every convex body is no greater than
$d - 1$. This would follow from equality of the two quantities
using the well-known result, proved by Fejes--Toth~\cite{fejes_toth}
in dimension two and by Bronshteyn--Ivanov~\cite{bi} in the general case,
that the (vertex) approximability of any convex body is no greater
than $(\thedim-1)/2$.

We show, using a slight modification of the technique
in Arya--da Fonseca--Mount~\cite{arya_da_fonseca_mount_journal},
that the Bronshteyn--Ivanov bound also holds for the flag approximability.

\begin{theo}
\label{thm:flag_approximability_upper_bound}
Let $\Omega \subset \R^\thedim$ be a convex body. Then
\begin{equation*}
\flagapprox(\Omega) \leq \frac {\thedim - 1} {2}.
\end{equation*}
\end{theo}

This allows us to deduce the entropy upper bound conjecture.
N.~Tholozan has also proved this conjecture recently using a different
method~\cite{tholozan}.

\begin{cor}
\label{cor:general_bound_on_entropy}
Let $\Omega \subset \R^\thedim$ be a convex body. Then
\begin{equation*}
\ent(\Omega) \leq \thedim - 1.
\end{equation*}
\end{cor}

For many Hilbert geometries, such as hyperbolic space, the volume of balls
grows exponentially. However, for some Hilbert geometries, the
volume grows only polynomially. In this case it is useful to make the
following definition.
Fix some notion of volume $\vol$.
The \textsl{asymptotic volume} of the Hilbert geometry on a $d$-dimensional
convex body $\Omega$ is defined to be
\begin{align*}
\assvol(\Omega) := \liminf_{R\to\infty}
                       \frac {\vol(\ball_\Omega(p,R))} {R^\thedim}.
\end{align*}
Note that, unlike in the case of the volume entropy, the asymptotic volume
depends on the choice of volume.
The first author has shown in~\cite{ver2012} that the asymptotic
volume of a convex body is finite if and only if the body is a polytope.

In the next theorem, we again see a connection appearing between
volume in Hilbert geometries and the number of maximal flags.
We denote by $\flags(\mathcal{P})$ the set of maximal flags of a polytope
$\mathcal{P}$.
Let $\dsimplex$ be a simplex of dimension $\thedim$.
Observe that $\flags(\dsimplex)$ consists of $(\thedim+1)!$ elements.

\begin{theo}
\label{thm:poly_volume_growth}
Let $\mathcal{P}$ be a convex polytope of dimension $\thedim$,
and fix some notion of volume $\vol$. Then,
\begin{align*}
\assvol(\mathcal{P})
    = \frac{\cardinal{\flags(\mathcal{P})}}
         {(\thedim+1)!} \assvol(\dsimplex).
\end{align*}
\end{theo}

An immediate consequence is that the simplex has the smallest asymptotic
volume among all convex bodies.
This was conjectured by Vernicos in~\cite{ver2012}.

\begin{cor}
\label{cor:smallest_asymptotic_volume}
Let $\Omega \subset \R^\thedim$ be convex body. Then,
\begin{equation*}
\assvol(\Omega) \geq \assvol(\dsimplex),
\end{equation*}
with equality if and only if $\Omega$ is a simplex.
\end{cor}

Another corollary is the following result, proved originally by Foertsch and
Karlsson~\cite{fk}.

\begin{cor}
\label{cor:normed_implies_simplex}
If a Hilbert geometry on a convex body $\Omega$ is isometric to
a finite-dimensional normed space, then $\Omega$ is a simplex.
\end{cor}

\section{Preliminaries}
\label{sec:preliminaries}

A \textsl{proper} open set in $\R^\thedim$ is an open set not containing a whole line.
A non-empty proper open convex set will be called a \textsl{convex domain}.
The closure of a bounded convex domain is called a \textsl{convex body}.

\subsection{Hilbert geometries}
A Hilbert geometry
$(\Omega,d_\Omega)$ is a convex domain $\Omega$ in $\R^\thedim$  with
the Hilbert distance $d_\Omega$ defined as follows.
For any distinct points $p$ and $q$ in $\Omega$,
the line passing through $p$ and $q$ meets the boundary $\partial \Omega$
of $\Omega$ at two points $a$ and $b$, labeled so that the line
passes consecutively through $a$, $p$, $q$, and $b$. We define
\begin{align*}
d_{\Omega}(p,q) := \frac{1}{2} \log [a,p,q,b],
\end{align*}
where $[a,p,q,b]$ is the cross ratio of $(a,p,q,b)$, that is, 
\begin{align*}
[a,p,q,b] := \frac{|qa|}{|pa|} \frac{|pb|}{|qb|} > 1,
\end{align*}
with $|xy|$ denoting the Euclidean distance between $x$ and $y$
in $\R^\thedim$.  If either $a$ or $b$ is at infinity, the corresponding ratio
is taken to be $1$. 

Note that the invariance of the cross ratio by a projective map implies
the invariance of $d_{\Omega}$ by such a map. In particular, since any convex
domain is projectively equivalent to a bounded convex domain, most
of our proofs will reduce to that case without loss of generality.

\subsection{The Holmes--Thompson and Busemann volumes}

Hilbert geometries are naturally endowed with
a $C^0$ Finsler metric $F_\Omega$ as follows.
If $p \in \Omega$ and $v \in T_{p}\Omega =\R^\thedim$
with $v \neq 0$, the straight line passing through $p$ and directed by $v$
meets $\partial \Omega$ at two points $p_\Omega^{+}$ and
$p_\Omega^{-}$~. Let $t^+$ and $t^-$ be two positive numbers such
that $p+t^+v=p_\Omega^{+}$ and $p-t^-v=p_\Omega^{-}$.
These numbers correspond to the time necessary to reach the boundary starting
at $p$ with velocities $v$ and $-v$, respectively. We define
\begin{align*}
F_\Omega(p,v) := \frac{1}{2} \biggl(\frac{1}{t^+} + \frac{1}{t^-}\biggr)
\quad \textrm{and} \quad F_\Omega(p , 0) := 0.
\end{align*}
Should $p_\Omega^{+}$ or $p_\Omega^{-}$ be at infinity,
the corresponding ratio will be taken to be $0$.


The Hilbert distance $d_\Omega$ is the distance induced by $F_\Omega$.
We shall denote by $B_\Omega(p,r)$ the metric ball of radius $r$
centered at the point $p\in \Omega$, and by $S_\Omega(p,r)$ the corresponding
metric sphere.

From the Finsler metric, we can construct two important Borel measures on
$\Omega$. 

The first is called the \textsl{Busemann} volume and is denoted by
$\bvol_\Omega$. It is actually the Hausdorff measure associated to
the metric space $(\Omega, d_\Omega)$; see Burago-Burago-Ivanov~\cite{bbi}, example~5.5.13.
It is defined as follows. For any $p \in \Omega$, let
\begin{align*}
\beta_\Omega(p) := \{v \in \R^\thedim ~|~ F_{\Omega}(p,v) < 1 \}
\end{align*}
be the open unit ball in $T_{p}\Omega = \R^\thedim$
of the norm $F_{\Omega}(p,\cdot)$,
and let $\omega_{\thedim}$ be the Euclidean volume of the open unit ball
of the standard Euclidean space $\R^\thedim$.
Consider the (density) function $h^B_\Omega \colon \Omega \longrightarrow \R$
given by
$h^B_\Omega(p)
    := \omega_{\thedim}/\operatorname{Leb}\bigl[\beta_\Omega(p)\bigr]$,
where $\operatorname{Leb}$ is the canonical Lebesgue measure
of $\R^\thedim$, equal to $1$ on the unit ``hypercube''.
Then for any Borel set $A$ in $\Omega$,
\begin{align*}
\bvol_\Omega(A) := \int_{A} h^B_\Omega(p) \,\ed{\operatorname{Leb}(p)}.
\end{align*}

The second, called the \textsl{Holmes--Thompson} volume, will be denoted
by $\htvol_{\Omega}$, and is defined as follows.
Let $\beta_\Omega^*(p)$ be the polar dual of $\beta_\Omega(p)$,
and let $h^H_{\Omega}\colon\Omega \longrightarrow \R$ be the density defined by
$h^H_{\Omega}(p)
    := \operatorname{Leb}\bigl[\beta_\Omega^*(p)\bigr]/\omega_\thedim$.
Then $\htvol_{\Omega}$ is the measure associated to this density.

In what follows, we will denote by $\htarea_{\Omega}$ and $\barea_\Omega$,
respectively, the $\thedim-1$-dimensional measures associated
to the Holmes--Thompson and Busemann measures.

\begin{lema}[Monotonicity of the Holmes--Thompson area]
\label{l:htmonotonicity}

Let $(\Omega, d_\Omega)$ be a Hilbert geometry in $\R^\thedim$.
The Holmes--Thompson area measure is
monotonic on the set of convex bodies in $\Omega$, that is, 
for any pair of convex bodies $K_1$ and $K_2$ in $\Omega$,
such that $K_1\subset K_2$, one has
\begin{equation}
\label{htmonotonicity}
\htarea_{\Omega}(\partial K_1) \leq \htarea_{\Omega}(\partial K_2).
\end{equation}
\end{lema}

\begin{proof}
If $\partial\Omega$ is $C^2$ with everywhere positive Gaussian curvature,
then the tangent unit spheres of the Finsler metric are quadratically convex.
According to \'Alvarez Paiva and
Fernandes~\cite[Theorem 1.1 and Remark 2]{alvarez_fernandez},
there exists a Crofton formula for the Holmes--Thompson area,
from which inequality~(\ref{htmonotonicity}) follows.
Such smooth convex bodies are dense in the set of all convex bodies
in the Hausdorff topology. By approximation,
it follows that inequality~(\ref{htmonotonicity}) is valid for any $\Omega$.
\end{proof}

The next result was essentially proved
by Berck-Bernig-Vernicos in~\cite[Lemma~2.13]{berck_Bernig_Vernicos}.

\begin{lema}[Co-area inequalities]
\label{l:co-area_with_cones}
Let $\Omega$ be a Hilbert geometry, with base point $o$, and let $\cone$ be a 
cone with apex $o$.
Then, for some constant $C>1$ depending only on the dimension~$\thedim$,
\begin{align*}
\frac{1}{C} \barea \bigl(S(R)\intersection\cone \bigr)
    \leq \frac{d}{d R} \bvol \bigl(B(R)\intersection\cone \bigr)
    \leq C \barea \bigl(S(R)\intersection\cone \bigr),
\end{align*}
for all $R\geq 0$.
\end{lema}

The results presented in this paper are actually mostly independent
of the definition of volume chosen; what really matters is that the volume one
uses  satisfies the following properties: 
continuity with respect to the Hausdorff pointed
topology, monotony with respect to inclusion, and invariance under
projective transformations. As a normalisation,
we need that the volume coincides with the standard one in
the case of an ellipsoid (see Vernicos~\cite{ver2014} for more details).

\subsection{Asymptotic balls}

Let $\Omega$ be a bounded open convex set.
For each $R\geq0$ and $y \in \Omega$, we call the dilation of $\Omega$
about $y$ by a factor $1-\exp(-2R)$ the \textsl{asymptotic ball}
of radius $R$ about $y$,
and we denote it by
\begin{align*}
\asympball_\Omega(y,R) := y + (1 - e^{-2R}) (\Omega - y).
\end{align*}
Some authors dilate by a factor $\tanh R$ instead,
but there is very little difference when $R$ is large.
By convention, we take $\asympball_\Omega(y,R)$ to be empty if $R\leq0$.
When there is no ambiguity, we sometimes omit mention of $\Omega$ or $y$
when denoting a ball or asymptotic ball.

The following lemma shows the close connection between asymptotic balls
and the balls of the Hilbert geometry.

\begin{lema}
\label{lem:asympotic_ball}
Let $\Omega$ be a bounded open convex set, containing a point $y$.
Assume that $\Omega$ contains the Euclidean ball of radius $l>0$ about $y$,
and is contained in the Euclidean ball of radius $L>0$ about $y$.
Then for all $R>0$ we have
\begin{equation*}
\asympball_\Omega \Bigl(y, R - \frac{1}{2}\log\bigl(1 + \frac{L}{l}\bigr) \Bigr)
    \subset \ball_\Omega(y,R)
    \subset \asympball_\Omega(y,R).
\end{equation*}
\end{lema}

\begin{proof}
Let $x\in\Omega$, and let $w$ and $z$ be the points in the boundary of $\Omega$
that are collinear with $x$ and $y$, labelled so that $w$, $x$, $y$, and $z$
lie in this order. Observe that $|xy| \leq L$ and $|yz| \geq l$.
Therefore,
\begin{align*}
1 \leq \frac{|xz|}{|yz|}
   = 1 + \frac{|xy|}{|yz|}
   \leq 1 + \frac{L}{l}.
\end{align*}
The point $x$ is in the ball $\ball_\Omega(y,R)$ if and only if
\begin{align*}
\log\frac{|wy|}{|wx|}\frac{|xz|}{|yz|} \le 2R,
\end{align*}
and is in the asymptotic ball $\asympball_\Omega(y,R)$ if and only if
\begin{align*}
\log\frac{|wy|}{|wx|} \le 2R.
\end{align*}
The result follows easily.
\end{proof}

Recall that the L\"owner--John ellipsoid of $\Omega$
is the unique ellipsoid of minimal volume containing $\Omega$.
By performing affine transformations, we may assume without loss of generality
that the L\"owner ellipsoid of $\Omega$ is the Euclidean unit ball
$\euclidean$.
It is known that $(1/\thedim)\euclidean$ is then contained in $\Omega$, that is,
\begin{equation*}
\frac{1}{\thedim}\euclidean\subset\Omega\subset\euclidean.
\end{equation*}
Thus, in this case the assumptions of Lemma~\ref{lem:asympotic_ball}
are satisfied with $L = 1$ and $l = 1/\thedim$.
A convex body will be said to be in \textsl{canonical form} if
its L\"owner--John ellipsoid is the unit Euclidean ball.

\section{Asymptotic volume and Flags}

\begin{quote}
In this section, we prove the study the asympototic volume of polytopes.
Our technique is to decompose the polytope into \textsl{flag simplices}.
We show that the asympototic volume of a flag simplex is independent of
the shape of the polytope, and depends only on the dimension.
Since there is one flag simplex for every maximal flag of the polytope,
our formula follows.
\end{quote}

\subsection{Flags and flag simplices}

Recall that to a closed convex set $K\subset \R^\thedim$  we can associate
an equivalence relation, where two points $a$ and $b$ are equivalent
if they are equal or if there exists an open segment $(c,d)\subset K$
containing the closed segment $[a,b]$.
The equivalence classes are called \textsl{faces}.
A face is called a \textsl{$k$-face} if the dimension
of its affine hull, that is, the smallest affine set containing it, is $k$.

A $0$-face is usually called an \textsl{extremal point},
or, in the case of convex polytopes, a \textsl{vertex}.
A \textsl{facet} is the relative closure of a face of co-dimension~1.

Thus defined, each face is an open set in its affine hull.
For instance, the segment $[a,b]$ in $\R$ admits three faces,
namely $\{a\}$, $\{b\}$, and the open segment $(a,b)$.
Notice that if  $K$ has non-empty interior, that is, if
$K\setminus\partial K\neq \emptyset$,
then its $\thedim$-dimensional face is its interior.

When a face $f$ is in the relative boundary of another face $F$, we write $f<F$.

\begin{defi}[Flag]
Let $\mathcal{P}$ be a closed convex $\thedim$-dimensional polytope.
A \textsl{maximal flag} of $\mathcal{P}$ is a $(\thedim + 1)$-tuple
$(f_0, ..., f_{\thedim})$ of faces of $\mathcal{P}$
such that each $f_i$ has dimension $i$,
and $f_0 < \dots < f_{\thedim}$.
\end{defi}

We denote by $\flags(\mathcal{P})$ the set of maximal flags of a polytope
$\mathcal{P}$.
We use $\lvert\cdot\rvert$ to denote the number of elements in a finite set.
The following formula will be useful. Let $\{F_i\}$ be the set of facets
of $\mathcal{P}$. So, each $F_i$ is a polytope of dimension $d - 1$.
We have that
\begin{align}
\label{eqn:flag_facet}
|\flags(\mathcal{P})| = \sum_i |\flags(F_i)|.
\end{align}

In this paper, a \textsl{simplex} in $\R^\thedim$ is the convex hull
of $\thedim+1$ projectively independent points, that is, a triangle in $\R^2$,
a tetrahedron in $\R^3$, and so forth.
If $\dsimplex$ is a simplex of dimension $\thedim$,
then $\flags(\dsimplex)$ consists of $(\thedim+1)!$ elements.

\begin{defi}[Flag simplex]
A simplex $\mathcal{S}$ is a \textsl{flag simplex} of a polytope $\mathcal{P}$
if there is a maximal flag $(f_0, ..., f_{\thedim})$ of $\mathcal{P}$
such that each of the faces $f_i$ contains exactly one vertex of $\mathcal{S}$.
\end{defi}

Let $\mathcal{P}$ be a convex polytope. Suppose that for each face of
$\mathcal{P}$ we are given a point in the face. Then, associated to each
maximal flag there is a flag simplex of $\mathcal{P}$, obtained by taking
the convex hull of the corresponding points.
Moreover, these flag simplices form a simplicial complex,
and their union is equal to $\mathcal{P}$.
We call this a \textsl{flag decomposition} of $\mathcal{P}$.
If each point is the barycenter of its respective face,
then the resulting flag decomposition
is just the well known \textsl{barycentric decomposition}.

\subsection{Flag simplices of simplices}

\begin{lema}
\label{lem:move_flags}
Let $T$ and $S$ be flag simplices of a $\thedim$-dimensional simplex
$\dsimplex$. Then, there exists a projective linear map $\phi$ leaving
$\dsimplex$ invariant, such that $\phi(T) \subset S$.
\end{lema}

\begin{proof}
We use induction on the dimension.
The induction hypothesis is that if $T$ and $S$ are flag simplices of
a $\thedim$-dimensional simplex $\dsimplex$, and $\{p_i\}$ is a finite
set of points in the interior of $\dsimplex$,
then there exists a projective linear map $\phi$ leaving
$\dsimplex$ invariant, such that the $\phi(T) \subset S$,
and the points $\{\phi(p_i)$\} are all contained in the interior of $S$.

The hypothesis  is clearly true in dimension $1$, since in this case
$\dsimplex$ is a closed interval, the flag simplices are closed segments
in $\dsimplex$ having one endpoint that co-incides with an endpoint of
$\dsimplex$ and the other in the interior, and the group of projective
linear maps is a one-parameter family that acts transitively on the interior
of $\dsimplex$.

Assume the hypothesis is true in dimension $d$, let $T$ and $S$ be
flag simplices of a $\thedim + 1$-dimensional simplex $\dsimplex$,
and let $\{p_i\}$ be a finite subset of the interior of $\dsimplex$.
Since the group of projetive linear maps acts transitively on the facets
of $\dsimplex$, we may assume that the flags associated to, respectively,
$T$ and $S$ have, as their facet, the same facet $F$ of $\dsimplex$.

\newcommand\union{\cup}
\newcommand\after{\circ}
Let $v$ be the vertex of $\dsimplex$ not contained in $F$,
and let $x$ be the vertex of $T$ not contained in $F$.
Project the points $\{p_i\}$ onto $F$ along rays emanating from $v$,
to get a set of points $\{q_i\}$. Project $x$ in the same way to get
a point $y$. By the induction hypothesis, there exists
a projective linear map $\phi_0$ on $F$ such that
$\phi_0(T \intersection F) \subset S\intersection F$,
and the point $y$ and all the points $\{q_i\}$
are mapped by $\phi_0$ into the relative interior of $S \intersection F$.
We can extend $\phi_0$ to a projective linear map on the whole of $\dsimplex$,
which we denote again by $\phi_0$.

There exists a $1$-parameter family of projective linear maps that fix $F$
and $v$. Amongs these maps, we can find one that maps
$x$ as close as we wish to $y$, and each of the points in $\{p_i\}$
as close as we wish to the corresponding point in $\{q_i\}$.
We choose such a map $\phi_1$ so that the image of $x$
and of each of the points $\{p_i\}$ is
in the interior of $\phi_0^{-1}(S)$.
So, the map $\phi := \phi_0 \after \phi_1$ maps $x$ and each of the points
$\{p_j\}$ into the interior of $S$. Since $T$ is the convex hull of $x$ and
$T \intersection F$, we have that $\phi(T) \subset S$.
\end{proof}

\begin{lema}
\label{lem:flag_simplex_in_simplex}
Consider the Hilbert geometry on a $\thedim$-dimensional simplex $\dsimplex$.
Let $S$ be a flag simplex of $\dsimplex$. Then for any $z$ in $\dsimplex$,
\begin{equation*}
\lim_{R\to\infty} \frac{1}{R^\thedim}
           \vol\bigl(\asympball(z,R) \intersection S\bigr)
   = \frac{1}{(\thedim+1)!} \assvol(\dsimplex).
\end{equation*}
\end{lema}

\begin{proof}
Because all simplices of the same dimension are affinely equivalent,
we may assume that ${\dsimplex}$ is a regular simplex with the origin $o$
as its barycenter.

Let $\bary$ be a barycentric flag simplex of ${\dsimplex}$.

A projective linear map leaving ${\dsimplex}$ invariant is an isometry
of the Hilbert metric on ${\dsimplex}$, and therefore preserves volume.
Combining this with the fact that
\begin{equation}
\label{eqn:nested_balls}
\ball\bigl(x,R-d(x,y)\bigr)
    \subset \ball(y,R)
    \subset \ball\bigl(x,R+d(x,y)\bigr),
\end{equation}
for any points $x,y\in {\dsimplex}$ and $R>0$, we get
\begin{equation}
\label{eqn:proh_no_different}
\lim_{R\to\infty} \frac{1}{R^\thedim}
           \vol\bigl(\ball(o,R) \intersection \phi(\bary)\bigr)
   =
\lim_{R\to\infty} \frac{1}{R^\thedim}
           \vol\bigl(\ball(o,R) \intersection \bary\bigr),
\end{equation}
for any projective linear map $\phi$ leaving ${\dsimplex}$ invariant.

From Lemma~\ref{lem:move_flags}, there exist projective linear maps $\phi_0$
and $\phi_1$ leaving ${\dsimplex}$ invariant,
such that $\phi_0(\bary) \subset S \subset \phi_1(\bary)$.
Combining this with~(\ref{eqn:proh_no_different}), we get
\begin{equation*}
\lim_{R\to\infty} \frac{1}{R^\thedim}
           \vol\bigl(\ball(o,R) \intersection S\bigr)
   =
\lim_{R\to\infty} \frac{1}{R^\thedim}
           \vol\bigl(\ball(o,R) \intersection \bary\bigr).
\end{equation*}

Denote by $\Pi$ the group of permutations of vertices of ${\dsimplex}$.
Observe that $\Pi$ has $(\thedim+1)!$ elements.
The group $\Pi$ acts on ${\dsimplex}$, leaving the center $o$ of
${\dsimplex}$ fixed.
We have that the union of the sets $\{\phi(\bary)\}_{\phi\in\Pi}$
is ${\dsimplex}$, and that the interiors of these sets
are pairwise disjoint. So, by symmetry,
\begin{equation*}
\lim_{R\to\infty} \frac{1}{R^\thedim}
           \vol\bigl(\ball(o,R) \intersection \bary\bigr)
   = \frac{1}{(\thedim+1)!} \assvol(\dsimplex).
\end{equation*}

The last step is to use~(\ref{eqn:nested_balls}) and
Lemma~\ref{lem:asympotic_ball} to get that 
\begin{equation*}
\lim_{R\to\infty} \frac{1}{R^\thedim}
           \vol\bigl(\asympball(z,R) \intersection S\bigr)
    = \lim_{R\to\infty} \frac{1}{R^\thedim}
           \vol\bigl(\ball(o,R) \intersection S\bigr).
\qedhere
\end{equation*}
\end{proof}

\subsection{Flag simplices of polytopes}

\begin{lema}
\label{lem:bounding_simplices}
Let $\mathcal{P}$ be a polytope, and let $S$ be a flag simplex of $\mathcal{P}$.
Then there exist simplices $U$ and $V$ satisfying
$U\subset \mathcal{P}\subset V$
such that $S$ is a flag simplex of both $U$ and of~$V$.
\end{lema}

\begin{proof}
We prove the existence of $U$ by induction on the dimension.
The one dimensional case is trivial, since here $\mathcal{P}$ is already
a simplex. So, assume the result holds in dimension~$\thedim$,
and let $\mathcal{P}$ be $\thedim+1$-dimensional.
Let $p$ be the vertex of $S$ that lies in the relative interior
of $\mathcal{P}$. The remaining vertices of $S$ form a flag simplex $S'$
of a facet of $\mathcal{P}$.
Applying the induction hypothesis, we get a simplex $U'$ contained in this
facet such that $S'$ is a flag simplex of $U'$.
It is not difficult to see that we may perturb $p$ in such a way as
to get a point $p'\in \mathcal{P}$ such that the simplex $U$ formed from $p'$
and $U'$ contains $p$ in its relative interior.
It follows that $U\subset \mathcal{P}$, and that $S$ is a flag simplex of $U$.

We also prove the existence of $V$ by induction on the dimension.
Again, the $1$-dimensional case is trivial.
As before, we assume the result holds in dimension~$\thedim$,
and let $\mathcal{P}$ be $\thedim+1$-dimensional.
Recall that $p$ is the vertex of $S$ that lies in the relative interior
of $\mathcal{P}$, and that the remaining vertices of $S$ form
a flag simplex $S'$ of a facet $F$ of $\mathcal{P}$.
Applying the induction hypothesis, we get a simplex $V'$ containing this
facet such that $S'$ is a flag simplex of $V'$.
Denote by $o$ the vertex of $S$ that is also a vertex of $\mathcal{P}$.
Without loss of generality we may assume that $o$ is the origin of
the vector space $\R^{\thedim+1}$.
Observe that if we multiply the vertices of $V'$ by any scalar $\alpha$ greater
than $1$, then $S'$ remains a flag simplex of $\alpha V'$.
Choose $q\in\R^{\thedim+1}$ and $\alpha>1$ such that
every vertex of $\mathcal{P}$ lies in the convex hull
\newcommand\conv{\operatorname{conv}}
\begin{align*}
V := \conv\{q, \alpha V'\}.
\end{align*}
Then, $\mathcal{P}\subset V$ and $S$ is a flag simplex of $V$.
\end{proof}

\begin{proof}[Proof of Theorem~\ref{thm:poly_volume_growth}]
Choose a flag decomposition of $\mathcal{P}$.
Let $x$ be the vertex that is common to all the flag simplices,
which lies in the interior of $\mathcal{P}$.

Let $S$ be any one of the flag simplices.
By Lemma~\ref{lem:bounding_simplices}, there are simplices $U$ and $V$
satisfying $U\subset \mathcal{P}\subset V$ such that $S$ is a flag simplex
both of $U$ and of~$V$. Hence,
\begin{align}
\label{eqn:volume_inequality}
\vol_U(X) \geq \vol_{\mathcal{P}}(X) \geq \vol_V(X),
\end{align}
for any measurable subset $X$ of the interior of $U$.
Observe that, for any $R>0$,
\begin{align}
\label{eqn:same_asymp_ball_intersection}
\asympball_U(x,R) \intersection S
   = \asympball_{\mathcal{P}}(x,R) \intersection S
   = \asympball_V(x,R) \intersection S.
\end{align}
Combining~(\ref{eqn:volume_inequality})
and~(\ref{eqn:same_asymp_ball_intersection})
with Lemma~\ref{lem:flag_simplex_in_simplex}, we get
\begin{align*}
\lim_{R\to\infty} \frac{1}{R^\thedim}
           \vol_{\mathcal{P}}\bigl(\asympball_{\mathcal{P}}(x,R) \intersection S\bigr)
   = \frac{1}{(\thedim+1)!} \assvol(\dsimplex).
\end{align*}
Using Lemma~\ref{lem:asympotic_ball}, we get from this that
\begin{align*}
\lim_{R\to\infty} \frac{1}{R^\thedim}
           \vol_{\mathcal{P}}\bigl(\ball_{\mathcal{P}}(x, R) \intersection S\bigr)
   = \frac{1}{(\thedim+1)!} \assvol(\dsimplex).
\end{align*}
But this holds for any flag simplex of the decomposition, and summing
over all the flags we get the result.
\end{proof}

\begin{proof}[Proof of Corollary~\ref{cor:smallest_asymptotic_volume}]
The first author proved in~\cite{ver2012}
that the asymptotic volume of a convex body is finite
if and only if it is a polytope.
The result follows because the simplex has fewer flags than any other polytope
of the same dimension.
\end{proof}

\begin{proof}[Proof of Corollary~\ref{cor:normed_implies_simplex}]
When one considers the Busemann volume,
the asymptotic volume of every normed space of a fixed dimension $d$
is the same, and is equal to $\assvol(\dsimplex)$
since the Hilbert geometry on a simplex is isometric to a normed space.
Hence $\assvol(\Omega) = \assvol(\dsimplex)$,
and the result follows from Corollary~\ref{cor:smallest_asymptotic_volume}.
\end{proof}

\section{A general bound on the flag complexity}
\label{sec:general_bound}

\begin{quote}
Here we prove Theorem~\ref{thm:flag_approximability_upper_bound},
that is, that the flag complexity of a $d$-dimensional convex body
is no greater than $(d-1)/2$.
\end{quote}

\newcommand\collectors{\mathcal{C}}
\newcommand\order{O}
\newcommand\witnesses{\mathcal{W}}

\renewcommand{\theenumi}{\roman{enumi}}

Our technique is to modify the proof of the main result
of Arya-da Fonseca-Mount~\cite{arya_da_fonseca_mount_journal}. In that paper, essentially the same result was
proved for the \textsl{face-approximability}, which is defined analogously
to the flag-approximability, but counting the least number of faces rather than
the least number of flags.

Their proof uses the witness-collector method.
Assume we have a set $S$ of points in $\R^\thedim$, a set $\witnesses$
of regions called \textsl{witnesses}, and a set $\collectors$ of regions called
\textsl{collectors}, satisfying the following properties.
\begin{enumerate}
\item
\label{itema}
each witness in $\witnesses$ contains a point of $S$ in its interior;
\item
\label{itemb}
any halfspace $H$ of $\R^\thedim$ either contains a witness $W\in\witnesses$,
or $H\intersection S$ is contained in a collector $C\in \collectors$;
\item
\label{itemc}
each collector $C \in \collectors$ contains some constant number of points
of $S$.
\end{enumerate}

We strengthen Lemma~4.1 of Arya-da Fonseca-Mount~\cite{arya_da_fonseca_mount_journal}.
In what follows, given a quantity $D$, any other quantity is said to be
$\order(D)$ if it is bounded from above by a multiple, depending only
on the dimension, of $D$.

\begin{lema}
\label{lem:flagcomplexity}
Given a set of witnesses and collectors satisfying the above properties,
the number of flags of the convex hull $P$ of $S$ is
$\order(|\collectors|)$.
\end{lema}

\begin{proof}
Take any facet $F$ of $P$, and let $H$ be the half-space whose intersection
with $P$ is $F$. As in the original proof, $H$ does not contain any witness,
for otherwise, by property~(\ref{itema}), it would contain a point of $S$ in its
interior. So, by~(\ref{itemb}), the intersection of $H$ and $S$ is contained in
some collector $C$.
Therefore, by~(\ref{itemc}), $F$ has at most $n$ vertices,
where $n$ is the number of points in each collector.

So, we see that each facet has at most $2^n$ faces,
and so has at most $(2^n)^\thedim$ flags,
since each flag can be written as an increasing sequence of $\thedim$ faces.

Also, the number of facets is at most $2^n |\collectors|$
since each facet has a different set of vertices, and this set is a
subset of some collector.

We deduce that the number of flags is at most
$(2^n)^{\thedim + 1} |\collectors|$.
\end{proof}

We conclude that the main theorem of~\cite{arya_da_fonseca_mount_journal} holds when
measuring complexity using flags instead of faces.

\begin{proof}[Proof of Theorem~\ref{thm:flag_approximability_upper_bound}.]
The proof follows that of the main result of~\cite{arya_da_fonseca_mount_journal},
but using Lemma~\ref{lem:flagcomplexity} above instead of Lemma~4.1 of that
paper.
\end{proof}

\section{Upper bound on the volume entropy}
\label{sec:upper_bound}

\begin{quote}
We show that the volume entropy of a convex body is no greater than twice the
flag approximability.
\end{quote}

\subsection{A uniform upper bound on the volume of a ball}

To prove the upper bound on the volume entropy, we will need to bound
the volume of balls of any radius in a polytopal Hilbert geometry
in terms of the number of flags of the polytope;
an asymptotic bound would be insufficient.
On the other hand, we will not be too concerned here with the exact dependence
on the radius---showing that it is polynomial will be enough.

We use $B(R)$ to denote the ball in a Hilbert geometry of radius $R$
and centered at $o$, and $S(R)$ to denote the boundary of this ball. We remind the reader that $\euclidean$ stands for the Euclidean unit ball.

\begin{lema}
\label{lem:new_flagareaub}
For each $\thedim\in\N$ and $0 < l \le 1$,
there exists a polynomial $p_{\thedim, l}$
of degree $\thedim$ such that the following holds.
Let $\mathcal{P}$ be a $\thedim$-dimensional polytope endowed with its Hilbert
geometry, satisfying $l.\euclidean \subset {\mathcal P} \subset \euclidean$.
Let $F$ be a facet of $\mathcal{P}$, and let $L$ be the cone with base $F$
and apex $o$.
Then,
\begin{align*}
\htvol \bigl(B(R)\intersection L \bigr) \leq p_{\thedim, l}(R)  \cardinal{\flags(F)},
\qquad\text{for all $R\geq 0$}.
\end{align*}
\end{lema}

\begin{proof} 
We will use induction on the dimension $\thedim$.
When $d=1$, there is only one Hilbert geometry, up to isometry.
In this case, $\htvol \bigl(B(R)\intersection L \bigr) = R/2$, and
$\cardinal{\flags(F)} = 2$, and so the conclusion is evident.

Assume now that the conclusion is true when the dimension is $d - 1$ and
$l$ is unchanged.

Using the co-area formula in Lemma~\ref{l:co-area_with_cones}, we get that
\begin{align*}
\frac{d}{d R} \htvol\bigl(B(R)\intersection\cone\bigr)
    \leq C \htarea\bigl(S(R)\intersection L\bigr),
\end{align*}
for some constant $C$ depending only on the dimension.

Denote the facets of $F$ by $\{F_i\}_i$. So, each $F_i$ is the closure of
a face of $\mathcal{P}$ of co-dimension $2$.
By~(\ref{eqn:flag_facet}),
\begin{align*}
\sum_i \cardinal{\flags(F_i)} = \cardinal{\flags(F)}.
\end{align*}
For each $i$, let $L_i$ be the
$\thedim-1$ dimensional cone with base $F_i$ and apex $o$.

\newcommand\asympsphere{\operatorname{AsS}}

Observe that, from Lemma~\ref{lem:asympotic_ball},
$B(R)\intersection L \subset \asympball(R) \intersection L$, for all $R\geq0$.
So, using the monotonicity of the Holmes--Thompson measure
(Lemma~\ref{l:htmonotonicity}), we get
\begin{align*}
\htarea\bigl(S(R)\intersection L\bigr)
    \leq \htarea\bigl(\asympsphere(R)\intersection L\bigr)
            + \sum_i \htarea\bigl(\asympball(R)\intersection L_i\bigr).
\end{align*}
Here $\asympsphere(R)$ is the boundary of the asymptotic ball of radius $R$
about $o$.
By the minimality of flats for the Holmes--Thompson
volume~\cite{alvarez_berck}, we have that
\begin{align*}
\htarea\bigl( \asympsphere(R)\intersection L \bigr)
    \leq \sum_i \htarea\bigl(\asympball(R)\intersection L_i \bigr).
\end{align*}

From Lemma~\ref{lem:asympotic_ball}, we have that
$\asympball(R) \subset \ball(R+c)$, where $c$ depends only on $l$.
Also, by the induction hypothesis,
\begin{align*}
\htarea\bigl(\ball(R+c) \intersection L_i\bigr)
    \leq p_{\thedim-1, l}(R+c) \cardinal{\flags(F_i)}.
\end{align*}

Putting all this together, we get that
\begin{align*}
\frac{d}{d R} \htvol(B(R)\intersection\cone)
    \leq 2 C p_{\thedim-1, l}(R + c) \cardinal{\flags(F)}.
\end{align*}
The result follows upon integrating.
\end{proof}

The two- and three-dimensional cases of the following theorem follow from 
Theorem $10$ in first author's paper~\cite{ver2014}.

\begin{theo}
\label{thm:new_growthpolytope}
For each $\thedim\in\N$ and $0 < l \le 1$,
there is a polynomial $p_{\thedim, l}$
of degree $\thedim$ such that,
for any $\thedim$-dimensional polytope ${\mathcal P}$ satisfying
$l.\euclidean \subset {\mathcal P} \subset \euclidean$,
we have
\begin{align*}
\htvol(B(R)) &\le p_{\thedim, l}(R) \cardinal{\flags(\mathcal{P})},
\qquad\text{for all $R\geq 0$}.
\end{align*}
The same result holds for the asymptotic balls.
\end{theo}

\begin{proof}
We will consider the metric balls; passing from these to the
asymptotic balls can be accomplished using Lemma~\ref{lem:asympotic_ball}.

Let $p_{\thedim, l}$ be the polynomial obtained from
Lemma~\ref{lem:new_flagareaub}. According to that lemma, for each facet $F$ of
$\mathcal{P}$ and for each $R>0$, we have
\begin{align*}
\htvol(B(R)\intersection L ) \le p_{\thedim, l}(R) \cardinal{\flags(F)},
\end{align*}
where $L$ is the cone with base $F$ and apex $o$.
Using~(\ref{eqn:flag_facet}) and summing over all the facets of
$\mathcal{P}$, we get the result.
\end{proof}

\subsection{The upper bound on the volume entropy}

\begin{figure}
\input{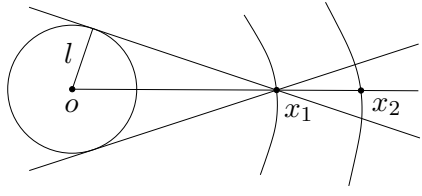}
\caption{Diagram for the proof of Lemma~\ref{lem:hausdorff_scaling}.}
\label{fig:hausdorff_scaling}
\end{figure}

\newcommand\ray{\operatorname{ray}}

\begin{lema}
\label{lem:hausdorff_scaling}
Let $\Omega_1$ and $\Omega_2$ be convex bodies within a Hausdorff distance
$\epsilon> 0$ of each other,
each containing the Euclidean ball $l\!\cdot\!\euclidean$ of radius $l > 0$
centered at the origin.
Then, $(1 / \lambda) \Omega_2 \subset \Omega_1$, with
$\lambda := 1 + \epsilon / l$.
\end{lema}

\begin{proof}
Consider a ray emanating from the origin, and let $x_1$ and $x_2$
be the intersections of this ray with the boundaries of $\Omega_1$ and
$\Omega_2$, respectively.
Let $l_1$ and $l_2$ be the distances from the origin to $x_1$ and $x_2$,
respectively, and suppose that $l_2 > l_1$.
Define the cone
\begin{align*}
  F := \bigl\{ x_1 + \alpha (x_1 - z)
    \mid \text{$\alpha > 0$ and $z \in l\!\cdot\! \euclidean$} \bigr\}.
\end{align*}
See Figure~\ref{fig:hausdorff_scaling}.
No point in the interior of the cone $F$ can be in $\Omega_1$.
However, the distance from $x_2$ to $\Omega_1$ is no greater than $\epsilon$.
This implies that the ball of radius $\epsilon$ around $x_2$ is not contained
in the interior of $F$. Looking at the sine of the angle subtended by $F$
at $x_1$, we see that $l / l_1 \leq \epsilon / (l_2 - l_1)$.
We deduce that
\begin{align*}
\frac{l_2 }{ l_1}
    = 1 + \frac{l_2 - l_1}{l_1}
    \leq 1 + \frac{\epsilon}{l}.
\end{align*}
The conclusion follows.
\end{proof}

\begin{lema}
\label{lem:upper_bound}
Let $\Omega$ be a convex body in $\R^\thedim$.
The volume entropy of $\Omega$ is no greater than twice its
flag approximability, that is,
\begin{equation*}
\ent(\Omega) \leq 2\flagapprox(\Omega).
\end{equation*}
\end{lema}

\begin{proof}
Without loss of generality, we may assume that $\Omega$ is in canonical form.
Let $R > 0$, and let $\epsilon > 0$ be such that $-2R = \log\epsilon$.
Let $P^*$ be a polytope approximating $\Omega$ within Hausdorff distance
$\epsilon$, having the least possible number $\flagnumber(\epsilon, \Omega)$
of maximal flags.
Write $\shrunkP := (1/\lambda) P^*$, where $\lambda := 1 + 2 \thedim \epsilon$.
When $\epsilon$ is small enough, both $\Omega$ and $P^*$
contain $(1/2d)\euclidean$, and so, by Lemma~\ref{lem:hausdorff_scaling},
\begin{align}
\label{eqn:shrink_factor}
(1/ \lambda^2) \Omega \subset \shrunkP \subset \Omega.
\end{align}

We will henceforth assume that $\epsilon$ is small enough
for this to be the case, and for $P$ to contain $(1/4d)\euclidean$.
Since $\Omega$ is in normal form, this implies that $\shrunkP$
satisfies the assumptions of Theorem~\ref{thm:new_growthpolytope},
with $l = 1/4d$.
Therefore, there exists a polynomial $p_\thedim$ of degree $\thedim$,
depending only on the dimension $\thedim$, such that
\begin{equation*}
\htvol_{\shrunkP}\bigl(\asympball_{\shrunkP}(o, R)\bigr)
    \leq \flagnumber(\epsilon, \Omega) p_\thedim(R).
\end{equation*}

From~(\ref{eqn:shrink_factor}),
\begin{equation*}
\htvol_{\Omega}(\cdot) \leq \htvol_{\shrunkP}(\cdot).
\end{equation*}

Observe that $((1 - \epsilon) / \lambda^2) \Omega$ is the asymptotic ball
of $\Omega$ of radius $R'$, where $- 2 R' = \log \epsilon'$,
with $1 - \epsilon' = (1 - \epsilon) / \lambda^2$.
Also, the asymptotic ball of $\shrunkP$ of radius $R$ is
$(1-\epsilon) \shrunkP$.
So, according to~(\ref{eqn:shrink_factor}),
\begin{equation*}
\asympball_\Omega(o, R') \subset \asympball_{\shrunkP}(o, R).
\end{equation*}

Finally, Lemma~\ref{lem:asympotic_ball} gives that
$\ball_\Omega(o, R') \subset \asympball_\Omega(o, R')$.

Putting all this together, we conclude that
\begin{equation*}
\frac{1}{R'} \log \htvol_{\Omega}\big(\ball_\Omega(o, R')\big) 
    \leq 2\frac{\log \big( \flagnumber(\epsilon, \Omega) p_d(R) \big)}
               {-\log\epsilon'}.
\end{equation*}
We now take the limit infimum as $R$ tends to infinity, in which case
$R'$ also tends to infinity, and $\epsilon$ and $\epsilon'$ tend to zero.
A simple calculation shows that, in this limit,
the ratio $\epsilon' / \epsilon$ converges to $2d + 1$. The result follows.
\end{proof}

\section{Lower bound on the volume entropy}

\begin{quote}
We show that the volume entropy of a convex body is no less than twice the
flag approximability.
\end{quote}

\begin{lema}
\label{lem:lower_bound}
Let $\Omega$ be a convex body in $\R^\thedim$.
Then, $2\flagapprox(\Omega) \leq \ent(\Omega)$.
\end{lema}

Our proof will be a modification of the method
used in Arya-da Fonseca-Mount~\cite{arya_da_fonseca_mount_journal}.
We start with a lemma concerning the centroid of a convex body, otherwise
known as its barycenter or center of mass.

\begin{lema}
\label{lem:bound_on_centroid}
Let $D$ be a convex body in $\R^\thedim$.
Let $p \in \partial D$ and $q\in D$ be such that the centroid $x$ of $D$
lies on the line segment $[pq]$. Then, $|px| \geq |pq|/(\thedim + 1)$.
\end{lema}

\begin{proof}
Let $h$ be a hyperplane tangent to $D$ at $p$.
The ratio $|px| / |pq|$ is minimized when $D$ is a simplex with a vertex at $q$
and all the other vertices on $h$.
\end{proof}

Recall the following definitions.
A \textsl{cap} $C$ of a convex body $\Omega$ is a non-empty intersection of $\Omega$
with a closed halfspace $H$. The \textsl{base} of the cap $C$
is the intersection of $\Omega$ with the hyperplane $h$ that bounds the halfspace.
An \textsl{apex} of $C$ is a point of $C$ of maximum distance from $h$.
Thus, the apexes of $C$ all lie in a hyperplane tangent to $\Omega$
and parallel to $h$.
The width of the cap is the distance from any apex to $h$.

Let $\Omega$ be a convex body containing the origin $o$ in its interior.
Consider the ray emanating from $o$ and passing through another point $x$.
We define the \textsl{ray-distance} $\ray(x)$ to be the distance from $x$
to the point where this ray intersects $\partial \Omega$.

\begin{lema}
\label{lem:distance_to_boundary}
Let $\Omega\subset \R^d$ be a convex body in canonical form.
Let $x$ be the centroid of the base of a cap of width $\epsilon$ of $\Omega$.
Then, the ray-distance $\ray(x)$ is greater than $C' \epsilon$,
for some constant $C'>0$ depending only on the dimension $\thedim$.
\end{lema}

\begin{proof}
Let $C$ be a cap of width $\epsilon$, and let $x$ be the centroid of its
base $D$. Let $z$ be an apex of $C$. So, $z$ is at distance $\epsilon$ from $h$,
the hyperplane defining the cap.

Consider the 2-plane $\Pi$ containing the points $o$, $x$, and $z$.
(If these points are collinear, then take $\Pi$ to be any 2-plane containing
them.)

The intersection of $D$ with $\Pi$ is a line segment.
Let $p$ and $q$ be the endpoints of this line segment.
Label them in such a way that the ray $ox$ intersects the line segment $pz$
at a point $w$.
See Figure~\ref{fig:distance_to_boundary}.
Think of $D$ as a convex body in $h$.
We get from Lemma~\ref{lem:bound_on_centroid}
that $|px| \geq |pq| / d$, since $x$ is the centroid of $D$.

\begin{figure}
\input{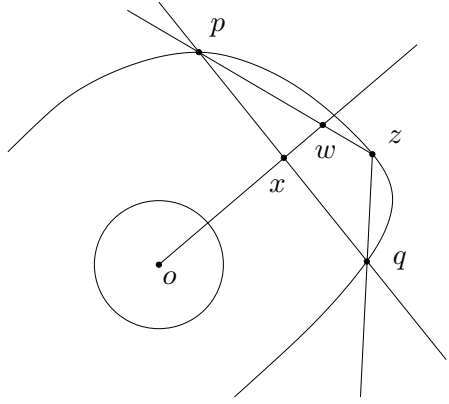}
\caption{Diagram for the proof of Lemma~\ref{lem:distance_to_boundary}.}
\label{fig:distance_to_boundary}
\end{figure}

We consider separately the cases where the angle $\angle pzq$ is acute and
where it is not.

\textsl{Case} $\angle pzq \leq \pi/2$.
Since $z$ is at distance at most $1$ from the origin, and
$\Omega$ contains the Euclidean ball $(1/d) \euclidean$,
the angle $\angle pzq$ must be at least
$A := 2\arcsin(1/\thedim)$. 
In the present case, this implies that $\sin \angle pzq$ is at least $\sin A$.
Observe that $|zq| \geq \epsilon$.
Two applications of the sine rule give
\begin{align*}
|xw| = |zq| \frac{|px|}{|pq|} \frac{\sin\angle pzq}{\sin \angle pwx}.
\end{align*}
We deduce that $|xw| \geq \epsilon \sin (A)/ d$

\textsl{Case} $\angle pzq \geq \pi/2$.
In this case there is a point $y$ between $p$ and $q$ such that
$\angle pzy = \pi/2$. Drop the perpendicular from $x$ to the line $pz$
to get a point $w'$ such that $\angle pw'x = \pi/2$.
Using similarity of triangles, we get
\begin{align*}
|xw|
    \geq |xw'| = \frac{|px| |yz|}{|py|}
    \geq \frac{|px| |yz|}{|pq|}
    \geq \frac{\epsilon} {d}.
\end{align*}

In both cases we have shown that $|xw|$ is at least $\epsilon$ times some
constant depending on the dimension.
The conclusion follows since $\ray(x) \geq |xw|$.
\end{proof}

The following is part of Theorem~2 of~\cite{ver2012}.

\begin{lema}
\label{lem:minimum_volume_ball}
For each dimension $d$, there is a constant $c$ such that
\begin{align*}
\htvol_\Omega\bigl(\ball_\Omega(x,R)\bigr) \ge c R^d,
\end{align*}
for each convex body $\Omega$, point $x\in\interior\Omega$, and radius $R>0$.
\end{lema}

Let $\Omega$ be a convex body containing a point $x$ in its interior.
The \textsl{Macbeath region} about $x$  is defined to be
\begin{align*}
M'(x) := x + \Bigl( \frac{1}{5} (\Omega - x) \intersection \frac{1}{5} (x - \Omega) \Bigr).
\end{align*}

Macbeath regions are related to balls of the Hilbert geometry as follows.

\begin{lema}
\label{lem:macbeath_vs_hilbert}
The Macbeath region $M'(x)$ about any point $x$ satisfies
\begin{align*}
\ball\Bigl(x, \frac{1}{2} \log \frac{6}{5}\Bigr)
    \subset M'(x)
    \subset \ball\Bigl(x, \frac{1}{2} \log \frac{3}{2}\Bigr).
\end{align*}
\end{lema}

\newcommand\dfunk{d_F}

\begin{proof}
Recall that the Funk distance between two points $p$ and $q$
is defined to be
\begin{align*}
\dfunk(p, q) := \log \frac{|pb|}{|qb|},
\end{align*}
where $b$ is as in the definition of the Hilbert metric
in section~\ref{sec:preliminaries}.
The Funk metric is not actually a metric since it is not symmetric.
Its symmetrisation is the Hilbert metric:
$d_{\Omega}(p,q) = (\dfunk(p,q) + \dfunk(q,p)) / 2$.

One can show that a point $y$ is in $M'(x)$ if and only if both
$\dfunk(x, y) \leq \log(5/4)$ and $\dfunk(y,x) \leq \log(6/5)$.
The conclusion follows.
\end{proof}

The following is a modification of
Lemma~3.2 of~\cite{arya_da_fonseca_mount_journal}.
The assumptions are the same; all that has changed is the bound on the number
of caps. The original bound was $\order(1/\delta^{(d-1)/2})$.

\begin{lema}
\label{lem:bound_on_number_of_caps}
Let $\Omega\subset R^\thedim$ be a convex body in canonical form.
Let $0 < \delta \leq \Delta_0 / 2$, where $\Delta_0$ is a certain constant
(see \cite{arya_da_fonseca_mount_journal}).
Let $\mathcal{C}$ be a set of caps each of width~$\delta$, such that the
Macbeath  regions $M'(x)$ centered at the centroids $x$ of the bases of these
caps are disjoint. Then,
\begin{align*}
|\mathcal{C}| = O\Bigl(\htvol\bigl(\asympball(o, R)\bigr)\Bigr),
\end{align*}
where $2R := -\log C \delta$, and $C$ is a constant depending only on the
dimension.
\end{lema}

\begin{proof}
Let $x$ be the centroid of the base of one of the caps in $\mathcal{C}$.
By Lemma~\ref{lem:distance_to_boundary}, the ray-distance satisfies
$\ray(x) \ge C' \delta$, where $C'$ is the constant appearing in that lemma.
Since $\Omega$ is contained in the unit ball,
this implies that $x\in\asympball(R')$, where $2 R' = - \log C' \delta$.
So, using Lemma~\ref{lem:asympotic_ball}, Lemma~\ref{lem:macbeath_vs_hilbert},
and Lemma~\ref{lem:asympotic_ball} again, we get that
the Macbeath region $M'(x)$ is contained within $\asympball(R)$,
where $2R = - \log C \delta$, with $C$ being some constant depending on the
dimension.

Combining Lemmas~\ref{lem:minimum_volume_ball}
and~\ref{lem:macbeath_vs_hilbert},
we get that there is a constant $C_1$ such that each Macbeath region $M'(x)$
has volume at least $C_1$.
A volume argument now gives that
$|\mathcal{C}| C_1 \leq \htvol(\asympball(R))$.
\end{proof}

We can now prove the lower bound on the volume entropy.

\begin{proof}[Proof of Lemma~\ref{lem:lower_bound}]
We may assume without loss of generality that $\Omega$ is in canonical form.

We follow the method of~\cite{arya_da_fonseca_mount_journal},
but using the bound in Lemma~\ref{lem:bound_on_number_of_caps}
on the number of non-intersecting Macbeath regions,
rather than that in Lemma~3.2 of~\cite{arya_da_fonseca_mount_journal}.
Given an $\epsilon >0$, this method produces a set of points $S$,
a set $\witnesses$ of witnesses, and a set $\collectors$ of collectors
satisfying the assumptions in section~\ref{sec:general_bound},
such that the convex hull of $S$ is an $\epsilon$-approximation of $\Omega$.
Furthermore, Lemma~\ref{lem:bound_on_number_of_caps} leads to the following
bound on the number of collectors:
\begin{align*}
|\collectors| \leq
\htvol\bigl(\asympball(R)\bigr) / C_1,
\end{align*}
where $2R := -\log C \delta$ and
$\delta := c_1\epsilon/\bigl(\beta\log(1/\epsilon)\bigr)$, for some constant
$c_1$ depending only on the dimension.

Since we are concerned with the flag-approximability, we must,
just as in the proof of Theorem~\ref{thm:flag_approximability_upper_bound},
use Lemma~\ref{lem:flagcomplexity} from section~\ref{sec:general_bound}
instead of Lemma~4.1 of~\cite{arya_da_fonseca_mount_journal}.
We get that the number $\flagnumber(\epsilon, \Omega)$ of flags
in the approximating polytope
is at most a fixed multiple $C_3 |\collectors|$ of $|\collectors|$.

Now let $\epsilon$ tend to zero.
Observe that $\log \delta / \log \epsilon$ converges to $1$.
So,
\begin{align*}
\flagapprox(\Omega)
     &= \liminf_{\epsilon\to 0} \frac{\log \flagnumber(\epsilon, \Omega)}
                                     {-\log \epsilon} \\
     &\leq \liminf_{R\to\infty}
        \frac{\log\big((C_3/C_1)
            \htvol(\asympball(R))\big)}{2R + \log C} \\
    &= \frac{1}{2} \ent(\Omega).
\qedhere
\end{align*}
\end{proof}

The proof of the main result of the paper is now complete.

\begin{proof}[Proof of Theorem~\ref{thm:entropy_flag_approximabilty}]
We combine Lemmas~\ref{lem:upper_bound} and~\ref{lem:lower_bound}.
\end{proof}

\bibliographystyle{plain}
\bibliography{enthilb}

\end{document}